\documentclass{amsart}
\usepackage{blindtext}
\usepackage[french,english]{babel}

\makeatletter
\renewcommand\part{%
   \if@noskipsec \leavevmode \fi
   \par
   \addvspace{4ex}%
   \@afterindentfalse
   \secdef\@part\@spart}

\def\@part[#1]#2{%
    \ifnum \c@secnumdepth >\m@ne
      \refstepcounter{part}%
      \addcontentsline{toc}{part}{\thepart\hspace{1em}#1}%
    \else
      \addcontentsline{toc}{part}{#1}%
    \fi
    {\parindent \z@ \raggedright
     \interlinepenalty \@M
     \normalfont
     \ifnum \c@secnumdepth >\m@ne
    \bfseries \partname\nobreakspace\thepart
       \par\nobreak
     \fi
      \bfseries #1%
     \par}%
    \nobreak
    \vskip 3ex
    \@afterheading}
\def\@spart#1{%
    {\parindent \z@ \raggedright
     \interlinepenalty \@M
     \normalfont
     \huge \bfseries #1\par}%
     \nobreak
     \vskip 3ex
     \@afterheading}
\makeatother

\usepackage{bm}

 \usepackage{upgreek}

\usepackage{amsopn, amsthm, amsgen, amscd,amsmath, amssymb}
\DeclareMathAlphabet{\mathbbm}{U}{bbm}{m}{n}
\usepackage{color}
\usepackage{tikz}
  \usepackage{graphicx}
\usepackage{epstopdf}
\usepackage[small,nohug,heads = LaTeX]{diagrams} \diagramstyle[labelstyle=\scriptstyle]

\definecolor{CadetBlue}{cmyk}{0.62, 0.57, 0.23, 0 }
\definecolor{black}{cmyk}{1, 0.5, 0, 0 }
\definecolor{RedViolet}{cmyk}{0.07, 0.9, 0, 0.34 }
\definecolor{SeaGreen}{cmyk}{0.69, 0, 0.5, 0}

\DeclareMathAlphabet{\mathpzc}{OT1}{pzc}{m}{it}

\newcommand{\R}{\mathbb R}
\newcommand{\C}{\mathbb C}
\newcommand{\D}{\mathbb D}

\newcommand{\F}{\mathbb F}

\newcommand{\N}{\mathbb N}
\newcommand{\PR}{\mathbb P}
\newcommand{\Q}{\mathbb Q}
\newcommand{\Z}{\mathbb Z}

\newcommand{\T}{\mathbb T}

\newcommand{\e}{\upvarepsilon}

\newtheorem{theo}{Theorem}
\newtheorem{lemm}{Lemma}
\newtheorem{prop}{Proposition}
\newtheorem{coro}{Corollary}
\newtheorem*{nntheo}{Theorem}

\newtheorem*{conj}{Conjecture}

\theoremstyle{definition}
\newtheorem{defi}{Definition}

\theoremstyle{remark}

\newtheorem{note}{Note}

\title[Quantum Drinfeld Modules I]{Quantum Drinfeld Modules I: Quantum Modular Invariant and Hilbert Class Fields}
\author{L. Demangos}
\address{Xi'an Jiaotong - Liverpool University, Department of Mathematical Sciences, Mathematics Building Block B, 111 Ren'ai Road, Suzhou Dushu Lake Science
and Education Innovation District, Suzhou Industrial Park, Suzhou, Peoples Republic of China, 215123}
\email{Luca.Demangos@xjtlu.edu.cn}
\author{T.M. Gendron}
\address{Instituto de Matem\'{a}ticas -- Unidad Cuernavaca, Universidad
Nacional Aut\'{o}noma de M\'{e}xico, Av. Universidad S/N, C.P. 62210
Cuernavaca, Morelos, M\'{e}xico}
\email{tim@matcuer.unam.mx}
\date{\today}
\begin{document}
\vspace{2cm}

 \maketitle

 \begin{abstract} This is the first of a series of two papers in which we present a solution to Manin's Real Multiplication program \cite{Man}  -- an approach to Hilbert's 12th problem for real quadratic extensions of $\Q$  --  in positive characteristic, using quantum analogs of the exponential
function and the modular invariant. 
In this first paper, we treat the problem of Hilbert class field generation.  If $k=\F_{q}(T)$ and $k_{\infty}$ is the analytic completion of $k$, we introduce the quantum modular
invariant 
\[ j^{\rm qt}: k_{\infty}\multimap k_{\infty}\]
as a multivalued, modular invariant function.  Then if $K=k(f)\subset k_{\infty}$ is a real quadratic extension of $k$ where $f$ is a quadratic unit, we show that
the Hilbert class field $H_{\mathcal{O}_{K}}$ (associated to $\mathcal{O}_{K}=$ integral closure of $\F_{q}[T]$ in $K$) is generated over $K$ by 
the product of the multivalues of $j^{\rm qt}(f)$.
 \end{abstract}

\tableofcontents

\section*{Introduction}

This paper studies a new and deep connection between diophantine approximation and algebraic number theory, implemented by
a multivalued function called the quantum modular invariant.  Specifically, we show that special multivalues of the quantum modular invariant may be used to generate Hilbert class fields of real quadratic global function fields, thus providing a new solution to Hilbert's 12th problem in this case.

The 12th problem of Hilbert, one of three on Hilbert's list which remains incontrovertibly open, concerns the search for analytic functions
whose special values generate
all of the abelian extensions of a finite extension $K/\Q$ (\cite{Schapp}, pages 249--250).  Particularly one is interested in explicit descriptions
of the Hilbert class field, the
ray class fields and the maximal abelian extension.  The problem extends naturally to the class of global fields i.e.\ it may be considered as well in the case of a finite extension $K/\F_{q}(T)$, where $\F_{q}$ is the finite field having $q=p^{n}$ elements, $p$ a prime.

Hilbert was motivated by the Theorem of Kronecker-Weber \cite{Neu}, which solves the case $K=\Q$ using the exponential function, as well as a conjectural treatment of the case $K$ a complex quadratic extension of $\Q$, which was soon after solved by
 Weber and his student Fueter \cite{Ser}, \cite{Sil}, in which the sought after analytic functions
 are the modular invariant and certain elliptic functions associated to elliptic curves having Complex Multiplication.  More specifically, if $\upmu\in K-\Q$
 is such that the elliptic curve $\T_{\upmu}=\C/ (\Z +\Z\upmu)$ has endomorphism ring $\mathcal{O}_{K}$ = the ring of $K$-integers,
  the Theorem of Weber-Fueter states that 
  \begin{itemize}
\item[1.] The Hilbert class field $H_{K} $ of $K$ satisfies 
\[ H_{K}=K(j(\upmu)),\]
where $j(\upmu )$ is the modular invariant of $\T_{\upmu}$.
\item[2.] The ray class field $K^{\mathfrak{M}} $ defined by the modulus $\mathfrak{M}\subset\mathcal{O}_{K}$ satisfies
\[ K^{\mathfrak{M}} = H_{K}( h_{\upmu}(t):\; t\in  \T_{\upmu}[\mathfrak{M}]   ),\]
where $ \T_{\upmu}[\mathfrak{M}]$ is the group of $\mathfrak{M}$ torsion points of $\T_{\upmu}$ and $h_{\upmu}$
is a multiple\footnote{In the case $j(\upmu)=0,12^{3}$, one must use multiples of $\wp_{\upmu}^{3}$ resp.\  $\wp_{\upmu}^{2}$.} of the Weierstra\ss\ function $\wp_{\upmu}$
on $\T_{\upmu}$.
\end{itemize}
 
In 2004, Yuri Manin \cite{Man} proposed the development of a parallel theory of {\it Real Multiplication} of quantum tori in order to treat the case of $K$ a real quadratic extension of $\Q$.  By definition, a {\it quantum torus} is a quotient of the form
\[ \T(\uptheta ) :=\R/(\Z +\uptheta\Z), \quad \uptheta\in\R-\Q;\]
it is the obvious analog of the complex torus introduced above in this context.
However, both $\T(\uptheta )$ and the moduli space of quantum tori \[ {\sf Mod}^{\rm qt}:={\rm PGL}_{2}(\Z )\backslash (\R-\Q)\] are noncommutative
spaces (non Hausdorff quotients) and so one is immediately confronted with a serious obstacle:  finding the appropriate analogs of modular invariant and elliptic function in this singular setting.  

In this paper and its sequel \cite{DGIII}, we give a solution to the Real Multiplication program in the case of $K$ a 
real\footnote{A quadratic extension $K/k$ is {\it real} if the place at $\infty$ splits completely, otherwise
it is called {\it complex}.} quadratic extension of the global field $k=\F_{q}(T)$ using quantum notions of the
modular invariant \cite{DGI}, \cite{DGII} and the exponential function.  

The quantum modular invariant was first introduced in the number field setting \cite{Ge-C}  
as a discontinuous, modular invariant and multi-valued function
\[ j^{\rm qt}:
 {\sf Mod}^{\rm qt}\multimap \R .\]
For any $\uptheta\in \R-\Q$ and $\upvarepsilon >0$, one defines first
the approximant, $j_{\upvarepsilon}(\uptheta )$, using Eisenstein-like series over the set of {\it $\upvarepsilon$ diophantine approximations}
\[ B_{\upvarepsilon}(\uptheta ) =\{ n\in\Z|\; |n\uptheta-m|<\upvarepsilon \text{ for some }m\in\Z \}.\]
Then $j^{\rm qt}(\uptheta )$ is defined
to be the set of limits of the approximants  $j_{\upvarepsilon}(\uptheta )$ as $\upvarepsilon\rightarrow 0$.  
See \S 1 of \cite{Ge-C}.  PARI-GP experiments (see the Appendix of \cite{Ge-C}) indicate that $j^{\rm qt}$ is multi-valued, and a more refined calculation due to Pink \cite{Pink} suggests
that $j^{\rm qt}(\uptheta )$ is a self-similar Cantor set when $\uptheta$ is quadratic.

\begin{conj}  Let $\uptheta\in\R-\Q$ be a fundamental quadratic unit and let $K=\Q (\uptheta )$; let $D$ be the fundamental discriminant of $K$.   Then 
$j^{\rm qt}(\uptheta )$ is a Cantor set, self-similar of order $D$ and
\[ H_{K}= K({\sf N}^{\rm avg}(j^{\rm qt}(\uptheta )))\]
where $H_{K}$ is the Hilbert class field of $K$ and ${\sf N}^{\rm avg}(j^{\rm qt}(\uptheta ))$ is a weighted product {\rm (}``multiplicative expectation''{\rm )} of the elements of $j^{\rm qt}(\uptheta )$.
\end{conj}

If verified, the conjecture would give a solution to the Hilbert class field part of Manin's Real Multiplication
 program.  
 In the present paper we formulate and prove the analog of this conjecture in the setting of function fields over finite fields.  In \cite{DGIII}, we 
 introduce a notion of quantum exponential function, and use the subset of its values corresponding to  ``quantum torsion points'' to give an explicit description of ray class fields.
 
 We now give a synopsis of the content of this paper.
Let $k=\F_{q}(T)$, $A=\F_{q}[T]$ and let
$k_{\infty}=\F_{q}((1/T))$ be the completion of $k$
with respect to the valuation $v_{\infty}(a)=-\deg_{T}(a)$. One views $k_{\infty}$ as the function field analog of the real numbers.  See \cite{Goss}, \cite{Thak} for basic notions of function field arithmetic.

We introduce the quantum modular invariant
\[ j^{\rm qt}: {\sf Mod}^{\rm qt}:={\rm GL}_{2}A\backslash (k_{\infty}-k) \multimap k_{\infty},
\]
following the same procedure used in the number field case.  That is, for $f\in k_{\infty}-k$, we start with the set of $\upvarepsilon$ diophantine approximations
\[ \Uplambda_{\upvarepsilon} (f) := \{ a\in A|\; | af -b|_{\infty}<\upvarepsilon \text{ for some $b\in A$}\}\]
(where $|\cdot|_{\infty}$ is the absolute value associated to $v_{\infty}$),
which in this setting is an $\F_{q}$ vector space.   We then define the approximant $j_{\upvarepsilon}(f)$ using $\Uplambda_{\upvarepsilon} (f) $ in place of the lattices occurring 
in the function field modular invariant (as defined in
 \cite{Gek}). Then the association
 \[ f\longmapsto j^{\rm qt}(f):= \lim_{\upvarepsilon\rightarrow 0} j_{\upvarepsilon} (f) \] defines a discontinuous, ${\rm GL}_{2}A$-invariant and
multivalued function; see \cite{DGI}, \cite{DGII} and \S \ref{Quantumjsection} of this paper.

Now fix $f\in k_{\infty}-k$ a fundamental quadratic unit, denote $K=k(f)$ and let $\mathcal{O}_{K}$ be the integral closure of $A$ in $K$.  
If $\deg_{T}(f)=d$ then the discriminant $D$ of $f$ satisfies $v_{\infty}(\sqrt{D})=-d$ and $ j^{\rm qt}(f)$ consists of precisely $d$ equi-distributed values (Theorem \ref{jfidthem}, Corollary \ref{injectivitycoro}
of this paper).  This result is the analog of the order of self-similarity predicted in the Conjecture above.
Let $H_{\mathcal{O}_{K}}$ be the Hilbert class field (in the sense of
Rosen \cite{Ros}) associated to $\mathcal{O}_{K}$.

\begin{nntheo} Let $f\in k_{\infty}-k$, $K=k(f)$ be as above.   
  Then 
\[ H_{\mathcal{O}_{K}} = K( {\sf N}(j^{\rm qt}(f) )),\]
where $ {\sf N}(j^{\rm qt}(f ))$ is the product of the $d$ elements of $j^{\rm qt}(f)$.
\end{nntheo}

To prove this theorem, we make use of the following description of $j^{\rm qt}(f)$ using ideals in a sub Dedekind domain of $\mathcal{O}_{K}$. 
Let $ \Upsigma_{K}$ be the curve over $\F_{q}$ associated
to $K$ and $\Upsigma_{K}\rightarrow \PR^{1}$ the morphism inducing the extension $K/\F_{q}(T)$.   Choose $\infty_{1}\in \Upsigma_{K}$ a point lying over $\infty\in \PR^{1}$
and let $A_{\infty_{1}}\subset\mathcal{O}_{K}$ be the sub Dedekind domain of functions regular
outside of $\infty_{1}$.   If  $H_{A_{\infty_{1}}}\supset H_{\mathcal{O}_{K}}$ is the Hilbert class field
associated to $A_{\infty_{1}}$,
write 
\begin{align}\label{CFTIso}  Z:= {\rm Gal}(H_{A_{\infty_{1}}}/H_{\mathcal{O}_{K}})\cong  {\rm Ker}\left( {\sf Cl}_{A_{\infty_{1}}}\longrightarrow {\sf Cl}_{\mathcal{O}_{K}}\right),
\end{align}
where ${\sf Cl}_{A_{\infty_{1}}}, {\sf Cl}_{\mathcal{O}_{K}}$ are the ideal class groups of $A_{\infty_{1}}$ resp.\ $\mathcal{O}_{K}$,
and the isomorphism in (\ref{CFTIso}) is that given by reciprocity.  The group $Z$ is cyclic of order $d$ (Proposition
\ref{KerProp} of \S \ref{HCFGenerationSection}) and
if we denote by $[\mathfrak{a}_{i}]$ the ideal class corresponding to $\upsigma_{i}\in Z$, $i=0,\dots ,d-1$ (in which $\mathfrak{a}_{0}=(f)$ defines the identity),
 then (Theorem \ref{jfidthem})
\[ j^{\rm qt}(f) = \{ j(\mathfrak{a}_{i} )|\; i=0,\dots ,d-1\} \subset H_{A_{\infty_{1}}},\]
where $ j(\mathfrak{a}_{i})$ is the $j$-invariant of the ideal class $[\mathfrak{a}_{i}]$ (defined in \S \ref{Quantumjsection}). 
 The Galois group $Z$
acts transitively on $j^{\rm qt}(f)$ making the latter a $Z$-torsor (see proof of Theorem \ref{HAinft1Thm}),
which implies ${\sf N}(j^{\rm qt}(f) )\in H_{\mathcal{O}_{K}}$.  Thus, to prove the Theorem, it suffices
to show that ${\sf N}(j^{\rm qt}(f) )^{\upsigma}\not={\sf N}(j^{\rm qt}(f) )$ for all $\upsigma\in {\rm Gal}(H_{\mathcal{O}_{K}}/K)$, see Theorem \ref{normdiff} of \S \ref{Injectivity}.   The proof of the latter is accomplished by way of a fine analysis of the absolute values of the zeta functions used in the definition
of the $j(\mathfrak{a}_{i})$.

It is important at this stage to compare the theory presented in this series of papers with the elegant theory of Hayes \cite{Hayes},
which makes an appearance in this work in the form of an essential
intermediate step -- showing that $j^{\rm qt}(f)$ consists of algebraic elements.  Hayes theory gives an explicit class field theory for function fields,
 however in this connection there are two
 points worth making:

\begin{itemize}
\item 
While Hayes theory gives, for each finite extension $L/k$, explicit descriptions of the class fields associated to a ``rank 1'' Dedekind domain 
\[ A_{P} = \{\text{functions regular in $\Upsigma_{L}-P$}\}, \quad P\in\Upsigma_{L}, \]  it does not give explicit descriptions of the traditional class fields 
classically considered in the number field setting:  the Hilbert class field and ray class fields associated 
to the integral closure $\mathcal{O}_{L}$ of $A$ in $L$.  The problem is that $\mathcal{O}_{L}$ has in general an infinite unit group, and so cannot be treated by Hayes' rank 1 techniques.
\item 
The generators of the Hilbert class fields $H_{A_{P}}$ provided by Hayes theory are not given as values of an analog of a modular function
(or indeed as values of any analytic function) except in the case where $L/k$ is a complex quadratic extension, where one has the exact counterpart of the theory of Complex Multiplication \cite{Gek}.  
\end{itemize}

The explicit class field theory described in this paper and its sequel, at least in the real quadratic case, may be seen as being somewhat closer in spirit to that called for by the 12th problem. 
The new object introduced, the quantum modular invariant, is available in characteristic zero, and, as the Conjecture above suggests,
offers a novel and practicable approach to the Real Multiplication program in its original characteristic zero formulation.  
Moreover, the Theorem we prove, being the positive characteristic analog of the Conjecture, offers evidence in support of the Conjecture's plausibility.

\vspace{3mm}

\noindent {\bf Acknowledgements:}  We would like to thank the Instituto de Matem\'{a}ticas (Unidad Cuernavaca) of Universidad Nacional Aut\'{o}noma de M\'{e}xico, as well
as the University of Stellenbosch, for their generous support of the first author during his postdoctoral stays at each institution.

\section{Analytic Notion of Quantum Drinfeld Module}

In this section we present the fundamental notion from which issues all of the essential constructions appearing in this paper and its sequel \cite{DGIII}: it may be informally referred to as the analytic notion of quantum Drinfeld module.  
In what follows we use basic notation already established in the Introduction.   For a review of the relevant background in function field arithmetic, see \cite{Goss}, \cite{Thak}, \cite{VS}.

For any
$x\in k_{\infty}$, denote by \[ |x|=q^{-v_{\infty}(x)}=q^{\deg_{T}(x)}\] the absolute value of $x$ and by $\| x\|$ the distance to the nearest element of $A$.  Note that $\| x\|<1$ and therefore, there exists a {\it unique}
$a\in A$ with $\| x\|=|x-a|$: $a$ is the ``polynomial part'' of $x$.  (The uniqueness follows from the non archimedean property of the absolute value.)   Now fix $f\in k_{\infty}$.
For any $\upvarepsilon>0$ the set 
\[ \Uplambda_{\upvarepsilon}(f) = \{ \uplambda \in A:\; \|  \uplambda f\|<\upvarepsilon\} \subset A\]
is an $\F_{q}$-vector space: this follows from the non-archimedean nature of the absolute value, see also Proposition 1 of \cite{DGI}.

In this paper and the sequel \cite{DGIII} we will be interested in studying the collection 
\[ \{ \Uplambda_{\upvarepsilon}(f)\}_{0<\upvarepsilon<1} \]  for $f$ a quadratic unit, about which much can be said.  Such an $f$ is the solution to 
an equation of the form
\begin{align}\label{defeqn} X^{2}-aX-b=0, \;\; a,b\in A,\;\; d:=\deg_{T} (a)>0,\; b\in \F_{q}^{\ast}.\end{align}   
Replacing $f$ by $cf$ if necessary, $c\in\F_{q}$, we may assume that $a$ is a monic polynomial.

Consider the sequence of monic polynomials
${\tt Q}_{n}\in A$ defined recursively by 
\[ {\tt Q}_{0}=1, {\tt Q}_{1}=a,\dots , {\tt Q}_{n+1}=a{\tt Q}_{n} +b{\tt Q}_{n-1}.\]
When $b=1$, this is the sequence of best approximations of $f$, see \cite{DGII}, \cite{Thak}.   Denote by $f^{\ast}$ the conjugate of $f$.  Without loss of generality we assume $|f|>|f^{\ast}|$, and then
 $|f|=|a|=q^{d}$ and $|f^{\ast}|=q^{-d}$.  Let $D=a^{2}+4b$ be the discriminant.
  
   For all $n$, we have (by an easy proof by induction) {\it Binet's formula}
\[{\tt Q}_{n} =\frac{ f^{n+1}-(f^{\ast})^{n+1}}{\sqrt{D}},\quad n=0,1,\dots ,\]
where the square root has been chosen so that $f=(a+\sqrt{D})/2$ for ${\rm char}(k)\not=2$ and otherwise $\sqrt{D}= a$.
Binet's formula gives
\begin{align}\label{Q_nerrorestimate} \| {\tt Q}_{n}f\| =  q^{ -(n+1) d } , \end{align}
since \[  \| {\tt Q}_{n}f\|=|{\tt Q}_{n}f - {\tt Q}_{n+1}| = 
 \frac{|f^{\ast} |^{n+1} |f-f^{\ast}|}{|\sqrt{D}|}  =
 |f^{\ast} |^{n+1}
=q^{ -(n+1) d }<1 .\]

The set \begin{align}\label{Tpowbasis} 
\mathcal{B}=\{ T^{d-1}{\tt Q}_{0}, \dots , T{\tt Q}_{0}, {\tt Q}_{0} ; T^{d-1}{\tt Q}_{1}, \dots , T{\tt Q}_{1}, {\tt Q}_{1};\dots  \}\end{align}
is a basis of $A$, since it has exactly one element for each possible polynomial degree.   
The order in which we have written the basis elements corresponds to decreasing errors.  Indeed,
\begin{align}\label{basiserrors}\| T^{l}{\tt Q}_{n}f \| =|T^{l}{\tt Q}_{n}f-T^{l}{\tt Q}_{n+1}|= q^{l-(n+1)d}<1.  \end{align}
In particular, (\ref{basiserrors}) shows that the map \begin{align}\label{errorbijection}  \mathcal{B}\longrightarrow q^{-\N},\quad  T^{l}{\tt Q}_{n} \longmapsto
\| T^{l}{\tt Q}_{n}f \| \end{align}
defines a bijection between $\mathcal{B}$ and the set of possible errors. 

Write 
 \[ \mathcal{B}(i) = \{ T^{d-1}{\tt Q}_{i},\dots , {\tt Q}_{i}\}\]  
 for the $i$th block of $\mathcal{B}$.
Furthermore, for $0\leq \tilde{d}\leq d-1$, denote
\[ \mathcal{B}(i)_{\tilde{d}} = \{T^{\tilde{d}}{\tt Q}_{i},\dots , {\tt Q}_{i}\}.\]

The following result appears in \cite{DGII}; being fundamental, we include it here as well.

\begin{lemm}\label{explicitLambda}  Let $l\in \{ 0,\dots ,d-1\}$
and write
\[ d_{l}=d-1-l.\]   Then 
\begin{align*} 
 \Uplambda_{q^{-Nd-l}}(f )={\rm span}_{\F_{q}} ( \mathcal{B}(N)_{d_{l}},\mathcal{B}(N+1),\dots   ) .
\end{align*}
\end{lemm}

\begin{proof}  Note that 
${\rm span}_{\F_{q}} ( \mathcal{B}(N)_{d_{l}},\mathcal{B}(N+1),\dots   ) \subset  \Uplambda_{q^{-Nd-l}}(f )$.
Moreover, by (\ref{basiserrors}), $ \Uplambda_{q^{-Nd-l}}(f )$ contains no other elements of $\mathcal{B}$. 
In view of the bijection (\ref{errorbijection}), no linear combination of the excluded basis elements could appear
in $ \Uplambda_{q^{-Nd-l}}(f )$.  
\end{proof}

We will now describe a Dedekind domain $A_{\infty_{1}}$ over which the $\Uplambda_{\upvarepsilon}(f)$ are ``almost'' modules.
Denote  
\[ K=k(f)\subset k_{\infty}.\]  Then $K$ is the function field of 
the projective curve $\Upsigma_{f}\subset\PR^{2}(\F_{q})$ defined by the equation (in the variables $X$ and $T$) \[ X^{2}-a(T)X-b=0.\]   
$\Upsigma_{f}$ has a singularity at $\infty$ for $d>1$, so we replace it by 
a smooth birationally equivalent model $\Upsigma$, which by construction is hyperelliptic, and whose genus $g$
satisfies the inequality (see \cite{Liu}, page 294, Proposition 4.24)
\[   2g+1 \leq \max \{ 2\deg(a(T)),\deg (b)  \}=2\deg(a(T))\leq 2g+2 . \]
From this it follows that $g=d-1$.   Let
\[ \uppi:\Upsigma\rightarrow \PR^{1}\] be a morphism inducing
the inclusion $k\hookrightarrow K$.  
  We assume that $\uppi$ is unramified over $\infty\in\PR^{1}$ which means that the valuation $v_{\infty}$  on $k$ has two extensions to $K$.  Picking
an extension amounts to picking a point $\infty_{1}\in \uppi^{-1}(\infty )$.  Denote by
\[ A_{\infty_{1}}\subset K \] the Dedekind domain of functions regular outside of $\infty_{1}$.  Since $\infty_{1}$ is not a branch point and $\Upsigma$ is hyperelliptic, it follows that $\infty_{1}$ is not a Weierstrass point.

Note that with respect to the extension  $v_{\infty_{1}}$ of the valuation $v_{\infty}$ to $K$ given by $\infty_{1}$, we have $v_{\infty_{1}}(f)=-d$.  Thus, as $\infty_{1}$ is not a Weierstra\ss\ point, its gap sequence is $1,\dots , g=d-1$, so $f$
is an element having pole of smallest order $=d$ at $\infty_{1}$.  It follows that
we may identify
\[ A_{\infty_{1}}= \F_{q}[f, fT,\dots , fT^{d-1}],\]  
see \cite{HKT}, page 188.
For example, when $d=1$, we obtain the familiar polynomial ring $A_{\infty_{1}}=\F_{q}[f]\cong \F_{q}[T]$, and when $d=2$, we have  
$ A_{\infty_{1}}=\F_{q}[f,fT] \cong \F_{q}[\wp,\wp']$,
where $\wp,\wp'$ are Weierstra\ss\ coordinates on the elliptic curve $\Upsigma$.

Note  that had we chosen $\infty_{2}\in \uppi^{-1}(\infty )$ instead of $\infty_{1}$ we would have
obtained
\[ A_{\infty_{2}}= \F_{q}[f^{-1}, f^{-1}T,\dots , f^{-1}T^{d-1}].\] 
This is because the Galois group of $K/k$ permutes the valuations associated to $\infty_{1}$ and $\infty_{2}$ and takes $f$ to a constant multiple
of $f^{-1}$.
In particular, $f$ has a zero of order $d$ at $\infty_{2}$.

Consider again the collection of $\F_{q}$-vector spaces $\{ \Uplambda_{\upvarepsilon}(f)\}$ introduced above.  
We will now show that after appropriately ``renormalizing'' the family $ \{ \Uplambda_{\upvarepsilon} (f)\}$ we 
obtain in the $\upvarepsilon\rightarrow 0$ limit a finite set of (analytic) rank 1 $A_{\infty_{1}}$-modules: a ``multivalued'' $A_{\infty_{1}}$-module.
For $\upvarepsilon_{N,l}:=q^{-dN-l}$, $l=0,\dots , d-1$, consider the rescaled vector space
\[ \hat{\Uplambda}_{\upvarepsilon_{N,l}}(f) := f^{-N}\sqrt{D}\Uplambda_{\upvarepsilon}(f) \]
Define as well the ideals
\[ \mathfrak{a}_{d-1-l}  := (f,fT,\dots ,fT^{d-1-l})\subset A_{\infty_{1}} .\]
\begin{prop}\label{renormalizedlimit}  For each $l=0,\dots , d-1$ we have 
\[ \lim_{N\rightarrow\infty}  \hat{\Uplambda}_{\upvarepsilon_{N,l}}(f) = \mathfrak{a}_{d-1-l}  \]
where the convergence is in the Hausdorff metric on subsets of $k_{\infty}$.
\end{prop}

\begin{proof}  By Lemma \ref{explicitLambda}, we have 
\[ \Uplambda_{N,l} ={\rm span}_{\F_{q}} (T^{d-1-l}{\tt Q}_{N},\dots , {\tt Q}_{N}, T^{d-1}{\tt Q}_{N+1},\dots ). \]
Using Binet's formula to replace each ${\tt Q}_{N+i}$ by $(f^{N+i+1}-(f^{\ast})^{N+i+1})/\sqrt{D}$, and the fact that $|f^{\ast}|<1$, in the limit $N\rightarrow \infty$
\[  f^{-N}\sqrt{D}{\tt Q}_{N+i} \rightarrow f^{i+1} \] uniformly in $i$.
\end{proof}

Thus we may regard the renormalized sequence of vector spaces $\{\hat{\Uplambda}_{\upvarepsilon}(f)\}$ as producing the multivalued limit
\[ \lim_{\upvarepsilon\rightarrow 0} \hat{\Uplambda}_{\upvarepsilon}(f) = \{\mathfrak{a}_{0},\dots ,\mathfrak{a}_{d-1} \} \]
which reinforces to a certain extent the quantum terminology.  This multivalued quality is a defining feature of all of the quantum objects/functions encountered in this work.
We end this section by developing this point in more detail: giving a rudimentary notion of (analytic) quantum Drinfeld module.

It is implicit in the statement of Proposition \ref{renormalizedlimit} that the elements of the collection $ \{ \hat{\Uplambda}_{\upvarepsilon} (f)\}$ define ``approximate'' $A_{\infty_{1}}$-modules.  
The following Lemma makes this precise.  In what follows we write for $X,Y\subset k_{\infty}$
\[ X\subset_{\upvarepsilon}Y \]
if ${\rm dist}(x,Y)<\upvarepsilon$ for all $x\in X$, where 
${\rm dist}(x,Y)=\inf_{y\in Y}|x-y|$.

\begin{lemm}\label{shiftlemma}  Fix $\upalpha \in A_{\infty_{1}}$, let $\upvarepsilon=\upvarepsilon_{N,l}$ and suppose that 
$\updelta=\updelta_{N,l} :=q^{-dN}\upvarepsilon$ satisfies
$\updelta<|\upalpha|^{-1}$.  Then
\begin{align*}
\upalpha \hat{\Uplambda}_{\upvarepsilon}(f)\subset _{|\upalpha|\updelta}    \hat{ \Uplambda}_{\upvarepsilon}(f) . \end{align*}
\end{lemm}

\begin{proof}  Let $\uplambda \in  \hat{\Uplambda}_{\upvarepsilon}(f)$.  Thus
\[ \uplambda =\sqrt{D}f^{-N}{\tt Q}_{N+i}T^{j}\]
where $i\geq 0$ and $0\leq j \leq d-1$ except when $i=0$ in which case $0\leq j \leq d-1-l$.  Then by Binet's formula
\begin{align*}  
|\uplambda -f^{i+1}T^{j} | & = | T^{j}(f^{\ast})^{N+i+1}f^{-N}| \\
& = q^{-(2N+i+1)d +j}\leq q^{-(2N+1)d +d-1-l}=q^{-(2Nd+l+1)} < \updelta.
 \end{align*}
Moreover the above inequality shows that every element of $\mathfrak{a}_{d-1-l}$ is within $\updelta$ of a unique element of $\hat{\Uplambda}_{\upvarepsilon}(f)$.
Now for $\upalpha \in A_{\infty_{1}}$, 
\[ |\upalpha\uplambda-\upalpha f^{i+1}T^{j}|< |\upalpha| \updelta\]
and since $|\upalpha\updelta|<1$, there exists a unique $\uplambda'\in \hat{\Uplambda}_{\upvarepsilon}(f)$ such that $|\uplambda'-\upalpha f^{i+1}T^{j}|<|\upalpha|\updelta$
as well.  Then $|\upalpha\uplambda-\uplambda'|<|\upalpha|\updelta$ and this proves the Lemma.
\end{proof}

Let ${\bf C}_{\infty}$ be the analytic completion of the algebraic closure of $k_{\infty}$.
For each $0<\updelta<1$ we consider the sub $\F_{q}$ vector space
\[ \mathfrak{Z}_{\updelta}:= \{ z\in {\bf C}_{\infty}|\; |z|<\updelta\}\] and define 
\[ \hat{\D}_{ \upvarepsilon, \updelta ,}(f) := {\bf C}_{\infty}/\left( \hat{\Uplambda}_{\upvarepsilon} (f) +\mathfrak{Z}_{\updelta}\right) . \]
Then for $\upalpha\in A_{\infty_{1}}$ and for each $\upvarepsilon=\upvarepsilon_{N,l}$ in which $\updelta_{N,l}=q^{-dN}\upvarepsilon<|\upalpha|^{-1}$,
 Lemma \ref{shiftlemma} implies that for any $\updelta$ with $\updelta_{N,l}<\updelta <|\upalpha|^{-1}$ 
there is a linear map 
\begin{align}\label{mapinducedbyalpha1}  \upalpha= \upalpha_{\upvarepsilon,\updelta}: \hat{\D}_{\upvarepsilon,\updelta }(f) \longrightarrow  \hat{\D}_{\upvarepsilon,|\upalpha|\updelta}(f) ,\quad z\longmapsto \upalpha z.\end{align}
If we have in addition $\updelta|\upalpha|<\updelta' <1$
we may compose the maps (\ref{mapinducedbyalpha1}) with the canonical projections 
\[  \hat{\D}_{\upvarepsilon,|\upalpha|\updelta}(f) \longrightarrow \hat{\D}_{\upvarepsilon,\updelta'}(f) , 
\]
which gives the family of maps 
\begin{align}\label{mapinducedbyalpha2}  \upalpha= \upalpha_{\upvarepsilon,\updelta,\updelta'}: \hat{\D}_{\upvarepsilon,\updelta }(f) \longrightarrow  \hat{\D}_{\upvarepsilon,\updelta'}(f),
\quad \upvarepsilon q^{-d}<\updelta <|\upalpha|^{-1},\;\;  \updelta|\upalpha|<\updelta' <1.
 \end{align}
We call such a family of maps an  {\bf approximate $A_{\infty_{1}}$-action}.   
The approximate $A_{\infty_{1}}$-action defined above satisfies evident analogs of the properties of a module.  That is,  for $\upalpha,\upbeta\in A_{\infty_{1}}$, 
there are commutative diagrams
\begin{diagram}
 \hat{\D}_{\upvarepsilon,\updelta}(f) & & \rTo^{\upalpha\upbeta} & & \hat{\D}_{\upvarepsilon,\updelta''}(f) \\
& \rdTo^{\upalpha} & & \ruTo^{\upbeta} & \\
& & \hat{\D}_{\upvarepsilon,\updelta'}(f)  & &
\end{diagram}
subject to the condition that all maps appearing belong to the families labeled respectively by $\upalpha,\upbeta,\upalpha\upbeta$.
Similarly, if we have
\[    \upalpha,\upbeta,\upalpha+\upbeta: \hat{\D}_{\upvarepsilon,\updelta}(f)\longrightarrow   \hat{\D}_{\upvarepsilon, \updelta'}(f)  \]
then 
\[ \upalpha\cdot z+\upbeta\cdot z =(\upalpha+\upbeta)\cdot z, \quad \forall z\in\hat{\D}_{\upvarepsilon,\updelta}(f). \]
We may refer to the above structure as an {\bf approximate $A_{\infty_{1}}$-module}.  The limiting ideal $\mathfrak{a}_{d-1-l}$ also
defines such a structure via 
\[  \D_{d-1-l,\updelta}:= {\bf C}_{\infty}/(\mathfrak{a}_{d-1-l}+ \mathfrak{Z}_{\updelta}),\] 
where $\upalpha\in A_{\infty_{1}}$ now acts on all of the $ \D_{d-1-l,\updelta}$.
This structure is the limit of the above in the sense that for any $\updelta$
and $\upvarepsilon=\upvarepsilon_{N,l}$ sufficiently small,
\[   \hat{\D}_{\upvarepsilon,\updelta}(f) =\D_{d-1-l,\updelta}.  \]


Consider now the simple quotients
\[  \hat{\D}_{\upvarepsilon}(f) := {\bf C}_{\infty}/\hat{\Uplambda}_{\upvarepsilon}(f) .\]
In view of Proposition \ref{renormalizedlimit},  we will
informally define the associated {\bf (analytic) quantum Drinfeld module} as 
\[ \hat{\D}^{\rm qt} (f):=`` \lim_{\upvarepsilon\rightarrow 0} \hat{\D}_{\upvarepsilon}(f)\text{''}  = \{ \D_{0}, \dots , \D_{d-1}\} , \quad \D_{i} := {\bf C}_{\infty}/\mathfrak{a}_{i}.\]
One thinks of $ \hat{\D}^{\rm qt} (f)$ as a ``multivalued'' Drinfeld module, whose ``multivalues'' are the rank $1$ Drinfeld modules $ \D_{i}$.
This definition is not yet rigorous, as we have not yet made precise what is meant by a ``point'' of $ \hat{\D}^{\rm qt} (f)$.  To do this, we must make formal the status of the limit, which will be done by defining connecting maps between the various $ \hat{\D}_{\upvarepsilon}(f)$.  The latter requires 
the algebraic complement of this analytic picture -- {\it the algebraic notion of quantum Drinfeld module} -- which will provide the notion of a ``multi-point'':  a Galois orbit of the shape
\[ z^{\rm qt}=\{ z_{i}|\; z_{i}\in \D_{i}\}.\] 
This will be done in the sequel to this paper \cite{DGIII}.

\section{The Quantum $j$-Invariant in Positive Characteristic}\label{Quantumjsection}

We begin by recalling the classical $j$-invariant in positive characteristic \cite{Gek}, \cite{Gek1}.  Let ${\bf C}_{\infty}$ be the analytic completion of the algebraic closure of $k_{\infty}$.  By a {\it lattice} $\Uplambda\subset {\bf C}_{\infty}$ is meant a discrete $A$-submodule of finite rank.  In what follows, we restrict to lattices
of rank 2 e.g.\ 
$\upomega\in\Upomega = {\bf C}_{\infty}- k_{\infty}$ defines the rank 2 lattice
$\Uplambda (\upomega) = \langle 1, \upomega\rangle_{A}=$ the $A$-module generated by $1,\upomega$.
Given a lattice $\Uplambda\subset  {\bf C}_{\infty}$, the {\it Eisenstein series} are
\[  E_{n}(\Uplambda) = \sum_{0\not=\uplambda\in\Uplambda} \uplambda^{-n},\quad n\in\N.\]
The {\it discriminant} is the expression
\[  \Updelta (\Uplambda) := (T^{q^{2}}-T)E_{q^{2}-1} (\Uplambda)+ (T^{q}-T)^{q}E_{q-1} (\Uplambda)^{q+1}\]
and setting
\[  g (\Uplambda) := (T^{q}-T)E_{q-1} (\Uplambda),\]
the (classical) {\it $j$-invariant} is defined
\[ j (\Uplambda) :=\frac{g (\Uplambda)^{q+1}}{\Updelta (\Uplambda)} = \frac{1}{\frac{1}{T^{q}-T}  + J  (\Uplambda)} \]
where
\[ J  (\Uplambda):=  \frac{T^{q^{2}}-T}{(T^{q}-T)^{q+1}} \cdot
\frac{E_{q^{2}-1} (\Uplambda)}{E_{q-1} (\Uplambda)^{q+1}}.
\]
When we restrict to $\Uplambda = \Uplambda (\upomega)$ we obtain a well-defined modular function 
\[  j: {\rm PGL}_{2}(A)\backslash\Upomega\longrightarrow  {\bf C}_{\infty}.\] 

We now introduce a quantum analog of the function $j$.  Let $f\in k_{\infty}$ (not necessarily quadratic).  The {\it $\upvarepsilon$-zeta function} of $f$ is:
\[ \upzeta_{f,\upvarepsilon}(m):=\sum_{\uplambda\in \Uplambda_{\upvarepsilon}(f)-\{0\}\atop
\uplambda\text{ monic }}\uplambda^{-m} , \quad m\in \N . \]
Then $ \upzeta_{f,\upvarepsilon}(m)$ is convergent, being a subsum of  
the {\it zeta function} of $A$  \[  \upzeta_{A}(m) = \sum_{a\in A\text{ monic}}a^{-m}.\]  In what follows, we will be interested in values of
zeta functions at integers of the form $m=n(q-1)$, $n\in\N$, which play the role of even integers for such zeta functions. 
In this case, as
\[ \sum_{c\in \F_{q}-\{ 0\}} c^{n(1-q)}=-1,\] we may write
\[ \upzeta_{f,\upvarepsilon}(n(q-1))= 
-\sum_{\uplambda\in \Uplambda_{\upvarepsilon}(f)-\{0\}}\uplambda^{n(1-q)}.  \]
Define
\begin{align*} \Updelta_{\upvarepsilon} (f) &:=-(T^{q^{2}}-T)\upzeta_{f,\upvarepsilon}(q^{2}-1)+ (T^{q}-T)^{q}\upzeta_{f,\upvarepsilon}(q-1)^{q+1}
 \end{align*} and 
\[g_{\upvarepsilon} (f) :=-(T^{q}-T)\upzeta_{f,\upvarepsilon}(q-1).\]
Then the {\it $\upvarepsilon$-modular invariant} of $f$ is defined
\[ j_{\upvarepsilon}(f):=\frac{{g_{\upvarepsilon}}^{q+1}(f)}{\Updelta_{\upvarepsilon}(f)} = \frac{1}{\frac{1}{T^{q}-T}  -J_{\upvarepsilon}(f) } \]
where
\begin{align}\label{Jep} J_{\upvarepsilon}(f) :=  \frac{T^{q^{2}}-T}{(T^{q}-T)^{q+1}} \cdot
\frac{\upzeta_{f,\e}(q^{2}-1)}{\upzeta_{f,\e}(q-1)^{q+1}}.
\end{align}

\begin{defi} The {\it quantum modular invariant} or {\it quantum j-invariant}  of $f$ is 
\[  j^{\rm qt}(f) := \lim_{\e\rightarrow 0} j_{\upvarepsilon}(f)   \subset k_{\infty}\cup \{ \infty \},\]
where by $ \lim_{\e\rightarrow 0} j_{\upvarepsilon}(f) $ we mean the {\it set of limit points} of convergent sequences $\{ j_{\e_{i}}(f)\}$, $\e_{i}\rightarrow 0$.  
\end{defi}

\begin{note} The reader is advised that no powers of modular invariants appear in this paper so that there should be no danger of confusing the quantum modular invariant $j^{\rm qt}$ with an expression such as 
``$j^{q^{t}}$''.
\end{note}

We remark that the association
$f\mapsto j^{\rm qt}(f) $ is non-trivially multivalued:  see \cite{DGII} or Corollary \ref{multicoro} of this paper.  It is invariant with respect to the projective action of ${\rm GL}_{2}A$ on $k_{\infty}$ (i.e.\
$j^{\rm qt}(Mf)=j^{\rm qt}(f)$ for all $M\in {\rm GL}_{2}A$ and $f\in k_{\infty}-k$, see \cite{DGI}), and so defines
a multivalued function
\[ j^{\rm qt}(f): k_{\infty}/{\rm GL}_{2}A\multimap k_{\infty}\cup \{\infty\}.\]
The domain $k_{\infty}/{\rm GL}_{2}A$ may be interpreted as the compactified moduli space of ``quantum tori'' $\T (f) =k_{\infty}/\langle1, f\rangle_{A}$, where $\langle1, f\rangle_{A}$ is the $A$-module
generated by $1,f$.

The following Theorems are the main results of \cite{DGI}, \cite{DGII}.
\begin{theo} For all $f\in k_{\infty}-k$ and $\upvarepsilon <1$, $|j_{\upvarepsilon}(f)|\in \{q^{q^{2}} ,q^{q^{2}+q-1}\} $.  
\end{theo}

\begin{proof}  Theorem 3 of \cite{DGI}, corrected in \cite{DGIe}.
\end{proof}

\begin{theo}\label{kcharact} $f\in k$ $\Leftrightarrow$ $ j_{\upvarepsilon}(f)=\infty$ for $\upvarepsilon$ sufficiently small $\Leftrightarrow$
$\infty\in j^{\rm qt}(f)$.
\end{theo}

\begin{proof}  Corollary 1 of \cite{DGI}.
\end{proof}

Since $k$ is dense in $k_{\infty}$, Theorem \ref{kcharact} implies that $j^{\rm qt}$ is not continuous.

\begin{theo} If $f\in k_{\infty}-k$ is quadratic, $\# j^{\rm qt}(f)<\infty$.
\end{theo}

\begin{proof}  Theorem 2 of \cite{DGII}.
\end{proof}

We now return to the case of $f$ a quadratic unit.
We will say that an element of $A_{\infty_{1}}$ is {\it monic} if it may be written in the form $c_{0}+c_{1}(T)f+\cdots +c_{k}(T)f^{k}$ where $c_{0}\in\F_{q}$,
$\deg c_{i}(T)\leq d-1$ for $1\leq i\leq k$ and $c_{k}(T)$ is a monic polynomial in $T$.
Then for any ideal $\mathfrak{a}\subset A_{\infty_{1}}$ we recall that the zeta function of $\mathfrak{a}$  is \cite{Goss} 
\[  \upzeta^{\mathfrak{a}}(n(q-1)) := \sum_{0\not= x\in \mathfrak{a}\text{ monic}} x^{n(1-q)}. \]
Define the $j$-invariant of the ideal $\mathfrak{a}$ as
\[  j(\mathfrak{a}) :=\frac{1}{\frac{1}{T^{q}-T}  -J(\mathfrak{a})},\quad 
J(\mathfrak{a}) :=  \frac{T^{q^{2}}-T}{(T^{q}-T)^{q+1}} \cdot
\frac{\upzeta^{\mathfrak{a}}(q^{2}-1)}{\upzeta^{\mathfrak{a}}(q-1)^{q+1}}.
\]
The $j$-invariant can also be defined using sums over all non-zero elements of $\mathfrak{a}$, where the $-$ sign in the denominator
in front of $J(\mathfrak{a})$ becomes a $+$ sign.  Then it is clear that $j(\mathfrak{a})$ only depends on the ideal class $[\mathfrak{a}]\in{\sf Cl}_{A_{\infty_{1}}}$
and $j$ defines a function
\[ j: {\sf Cl}_{A_{\infty_{1}}}\longrightarrow k_{\infty}.\]

Consider the ideals introduced in the previous section
\[ \mathfrak{a}_{i} = (f, fT,\dots, fT^{i})\subset A_{\infty_{1}},\quad i=0,\dots, d-1. \]
Note that $\mathfrak{a}_{i}$ is non principal for $i\not=0$. 

\begin{lemm} For $i=1,\dots ,d$, $\mathfrak{a}_{d-1}^{i}=\mathfrak{a}_{d-i}$.  In particular, the set of classes $\{ [ \mathfrak{a}_{i}]\}$
forms a cyclic subgroup of order $d$ of ${\sf Cl}_{A_{\infty_{1}}}$ generated by $[\mathfrak{a}_{d-1}]$.
\end{lemm}

\begin{proof}  In order to ease notation we will prove the statement in the case where $f^{2}=fT^{d}+1$; the proof in the general case is identical in form.  First we note
that  $\mathfrak{a}_{d-1}^{2}$ has generating set
\[  \{ f^{2} T^{j}\}_{j=0}^{2d-2} .\]
The generators $f^{2}T^{d},\dots , f^{2}T^{2d-2}$ may be rewritten 
\[   f^{2}T^{d}= f^{3}-f,\dots , f^{2}T^{2d-2}=(f^{3}-f)T^{d-2}  .\]
Since $f^{3},\dots ,f^{3}T^{d-2}\in \mathfrak{a}_{d-1}^{2}$, it follows that $f,\dots ,fT^{d-2}\in \mathfrak{a}_{d-1}^{2}$, hence $\mathfrak{a}_{d-2}\subset \mathfrak{a}_{d-1}^{2}$. The remaining generators $f^{2},\dots ,f^{2}T^{d-1}$
are clearly in $\mathfrak{a}_{d-2}$, so it follows that $\mathfrak{a}_{d-2}= \mathfrak{a}_{d-1}^{2}$.  Inductively, by a similar argument, we have 
\[ \mathfrak{a}_{d-1}^{i+1} = \mathfrak{a}_{d-i}\mathfrak{a}_{d-1}=\mathfrak{a}_{d-i-1}. \]
\end{proof}

\begin{theo}\label{jfidthem} Let $f$, $\mathfrak{a}_{i}$ be as above.  Then
\[ j^{\rm qt}(f) = \left\{ j (\mathfrak{a}_{i} ) \right\}_{i=0}^{d-1}.\]
\end{theo}

\begin{proof}   Let $\upvarepsilon=q^{-Nd-l}$ where $0\leq l\leq d-1$.  Every possible value $<1$ of the absolute value is of this form.  Abbreviate 
$ \Uplambda_{N,l}:=\Uplambda_{q^{-Nd-l}}(f )$ and write
\begin{align}\label{JTilde} \tilde{J}_{\upvarepsilon}(f)=
\frac{\upzeta_{f,\e}(q^{2}-1)}{\upzeta_{f,\e}(q-1)^{q+1}}=
\frac{ \sum_{0\not=a\in \Uplambda_{N,l}\text{ monic}}a^{1-q^{2}}   }{\left(\sum_{0\not=a\in \Uplambda_{N,l}\text{ monic}}a^{1-q}  \right)^{1+q}} .\end{align}
By Lemma \ref{explicitLambda}, we have 
\[ \Uplambda_{N,l} ={\rm span}_{\F_{q}} (T^{d-1-l}{\tt Q}_{N},\dots , {\tt Q}_{N}, T^{d-1}{\tt Q}_{N+1},\dots ). \]
Since the numerator and denominator of (\ref{JTilde})
have homogeneous degree $1-q^{2}$,  we may replace $\Uplambda_{\upvarepsilon}(f)$ by $\tilde{\Uplambda}_{\upvarepsilon}(f)=\sqrt{D}f^{-N}\Uplambda_{\upvarepsilon}(f)$
(e.g. by dividing the numerator and denominator by $(D^{-1/2}f)^{(1-q^{2})N}$).   By Proposition \ref{renormalizedlimit}, the limit $N\rightarrow\infty$ in the above rewriting of (\ref{JTilde})
yields the value
\[
\frac{ \sum_{0\not=x\in \mathfrak{a}_{d-1-l}\text{ monic}}a^{1-q^{2}}   }{\left(\sum_{0\not=x\in  \mathfrak{a}_{d-1-l}\text{ monic}}a^{1-q}  \right)^{1+q}} = \tilde{J}
(\mathfrak{a}_{d-1-l}), \]
where \[ \tilde{J}
(\mathfrak{a}_{d-1-l}):= \upzeta^{\mathfrak{a}_{d-1-l}}(q^{2}-1)/\upzeta^{\mathfrak{a}_{d-1-l}}(q-1)^{q+1}.\]
\end{proof}


\section{Injectivity}\label{Injectivity}

Fix $\mathfrak{a}\subset A_{\infty_{1}}$ a non-principal ideal. We begin by picking a convenient representative of $[\mathfrak{a}] \in{\sf Cl}_{A_{\infty_{1}}}$ as well as specifying for the representative
an $\F_{q}$-vector space basis.  Write $\mathfrak{a}=(g,h)$ and suppose that $\deg (g)\leq \deg (h)$.  Without loss of generality, we may assume
 $g$ has the smallest degree of all monic non-zero
elements\footnote{Since $A_{\infty_{1}}$ is Dedekind, for any non-zero $x\in\mathfrak{a}$, there exists $y\in\mathfrak{a}$ with $\mathfrak{a}=(x,y)$.
See for example Theorem 17 of \cite{Mar}, page 61.} of $\mathfrak{a}$.  Since $\mathfrak{a}$ is not principal, $h/g$ does not belong to $A_{\infty_{1}}$.
 Consider an $\F_{q}$-vector space basis of $\mathfrak{a}$ \begin{align}\label{vecbasis}\{ a_{0}=g,a_{1},\dots \},\end{align} 
 where say $a_{i_{0}}=h$ for some $i_{0}$.  We may assume that the elements
 \begin{align*} fg, fTg, fT^{2}g,\dots , fT^{d-1}g,f^{2}g,f^{2}Tg,\dots \end{align*} are elements of this basis; note however that they do not form a complete list of basis elements, since $h$ is not among them.    We will also assume that 
the map $a_{i}\mapsto \deg (a_{i})$ is injective and that $\deg (a_{i})<\deg (a_{j})$ for $i<j$.  And finally, we may also assume that
the $a_{i}$ are monic as polynomials in $f$ (in the sense described in the previous section).
Consider the fractional ideal in $[\mathfrak{a}]$
\begin{align}\label{astarlist} \mathfrak{a}^{\star}:=g^{-1}\mathfrak{a} = \langle \upalpha_{0}=1,\upalpha_{1}, \upalpha_{2},\dots ,\upalpha_{n},f, fT,\dots  \rangle_{\F_{q}},\quad \upalpha_{i}:=a_{i}/g.\end{align}
  Since $\mathfrak{a}$ is not principal, we have 
$ \mathfrak{a}^{\star}\supsetneq (1)$.  

\begin{lemm}\label{basislemma} The second generator $h$ of $\mathfrak{a}=(g,h)$ may be chosen so that the
$\upalpha_{1},\dots ,\upalpha_{n}\in K-A_{\infty_{1}}$ appearing in (\ref{astarlist}) satisfy \[ \quad 1<|\upalpha_{1}|<\cdots <|\upalpha_{n}|  <|f|.\]
\end{lemm}

\begin{proof}  By construction of the basis (\ref{vecbasis}), we will be done once we show that we may take the generator $h$ so that
 $1<|\upalpha_{i_{0}}|=|h/g|<|f|$.   Suppose otherwise i.e.\ that $\deg (h/g)\geq d $.  Then there exists $a\in A_{\infty_{1}}$ so that
\[   h'=ag +h\]
has degree $<\deg (h)$. Then $\mathfrak{a}=(g, h')$ and $\deg (h'/g)<\deg (h/g)$.  Replace $h$ by $h'$ and proceed inductively:  we reduce to the case $\deg (g)+d\geq \deg (h)$.  If we have equality, 
we take $h'=cfg+h$, where $c\in\F_{q}$ is a suitable constant so that $\deg (h')<\deg (h)$.  We cannot continue any further since
there are no non-constant elements of $A_{\infty_{1}}$ having degree $<d$.   Note that $\deg (h')>\deg (g)$ since $(g,h')=\mathfrak{a}$ is non principal
and $g$ is by assumption the monic element of smallest degree.  
\end{proof}

For the remainder of this section we will work with $\mathfrak{a}^{\star}=g^{-1}\mathfrak{a}$, renaming it $\mathfrak{a}$ and 
fixing the basis satisfying the conditions of Lemma \ref{basislemma}.  That is,
\[\mathfrak{a}=\langle 1,\upalpha_{1},\dots \rangle_{\F_{q}},\]
where $\upalpha_{i}=a_{i}/g$, as specified in the paragraph before Lemma \ref{basislemma}.
An element $x\in \mathfrak{a}$ is called monic if the integral element $gx$ is.

For $n\in\N$, consider the zeta function
\[ \upzeta^{\mathfrak{a}}(n(q-1)) =\sum_{0\not= x\in \mathfrak{a}\text{ monic}} x^{n(1-q)} = 
1+\sum_{i=1}^{\infty} \Upomega_{i}^{\mathfrak{a}}(n(q-1))\]
where 
\begin{align*} \Upomega_{i}^{\mathfrak{a}}(n(q-1)) & =  \sum_{c_{0},\dots ,c_{i-1}\in \F_{q}} \left(c_{0}+c_{1}\upalpha_{1}+\dots + c_{i-1}\upalpha_{i-1} + \upalpha_{i}\right)^{n(1-q)} \\
&:= \sum_{\vec{c}\in\F_{q}^{i}} \left(\vec{c}\cdot\vec{\upalpha}_{i-1} +\upalpha_{i}\right)^{n(1-q)}
  \end{align*}
  and where $\vec{\upalpha}_{i-1}=(\upalpha_{0}=1,\upalpha_{1},\dots ,\upalpha_{i-1})$.

\begin{lemm}\label{omegalemma}  Let $\mathfrak{a}$, $ \Upomega_{i}^{\mathfrak{a}}(q^{n}-1)$ be as above.  
\begin{enumerate}
\item When $i=1$, 
\[   \Upomega_{1}^{\mathfrak{a}}(q^{n}-1)  =\frac{\upalpha_{1}^{q^{n}}-\upalpha_{1}}{\prod_{c\in\F_{q}}(c+\upalpha_{1}^{q^{n}})}\quad\text{and}\quad 
 \left| \Upomega_{1}^{\mathfrak{a}}(q^{n}-1)\right| = |\upalpha_{1}|^{q^{n}(1-q)}.\]
\item  For all $i\geq 1$,
\begin{align*} \left|  \Upomega_{i}^{\mathfrak{a}}(q^{n}-1)\right| &  
\leq |\upalpha_{i}|^{q^{n}(1-q)}.
\end{align*}
\end{enumerate}
\end{lemm}

\begin{proof} 
Write $\upalpha=\upalpha_{1}$.  Then 
\begin{align}\label{lemmafraction} \Upomega_{1}^{\mathfrak{a}}(q^{n}-1)=
\sum_{c}\frac{c+\upalpha}{c+\upalpha^{q^{n}}}=
 \frac{\sum_{c} (c+\upalpha)\prod_{d\not=c}(d+\upalpha^{q^{n}})}{\prod_{c}(c+\upalpha^{q^{n}})} .
\end{align}
Denote by $s_{i}(c)$ the $i$th elementary symmetric function on $\F_{q}-\{ c\}$.  Thus, $s_{0}(c)=1$, $s_{1}(c)=\sum_{d\not=c}d$, etc. Then the numerator of (\ref{lemmafraction}) may be written as 
\begin{align*}
\sum_{j=0}^{q-1}\bigg(\sum_{c} (c+\upalpha)s_{q-1-j}(c)\bigg)\upalpha^{jq^{n}} .\\
\end{align*}
First note that there is no constant term, and the coefficient of $\upalpha$ is $\prod_{d\not=0}d=-1$.
Now
\[  \sum_{c}cs_{q-2}(c)= \sum_{c\not=0}cs_{q-2}(c)=\sum_{c\not=0} c\prod_{d\not=0,c}d =\sum_{c\not=0}c\cdot (- c^{-1})=(q-1)(-1)=1,\]
which is the coefficient of $\upalpha^{q^{n}}$.
Moreover, $\sum_{c}s_{q-2}(c)=  s_{q-2}(0)  -\sum_{c\not=0}c^{-1}=0$, so the $\upalpha^{q^{n}+1}$ term vanishes.  
For $i<q-2$, we claim that
\begin{align*}
\sum_{c} cs_{i}(c)=0= \sum_{c} s_{i}(c). \end{align*}
When $i=0$,  $s_{0}(c)=1$ for all $c$, 
the terms $\upalpha^{q^{n}(q-1)}$, $\upalpha^{q^{n}(q-1)+1}$ have coefficients $\sum_{c}c=\sum_{c}1=0$ and so vanish.
When $i=1$, we have $q>3$,
so $s_{1}(c)=-c$ and  \[  \sum_{c}s_{1}(c)=-\sum_{c}c=0=-\sum_{c}c^{2} =\sum_{c}cs_{1}(c)\]
since the sums occurring above are power sums over $\F_{q}$ of exponent $1,2<q-1$.
For general $i<q-2$, we have $q>i+2$ and
\[ \sum_{c} s_{i}(c) = \sum_{c} P(c)\]
where $P(X)$ is a polynomial over $\F_{q}$ of degree $i<q-2$.   Hence  $\sum_{c} P(c)= \sum_{c} cP(c)=0$, since again, these are sums of powers of $c$ of exponent less than $q-1$.
Thus 
the numerator is $ \upalpha^{q^{n}}-\upalpha $
and the absolute value claim follows immediately.
When $i>1$, for each $\vec{c}$, let $\vec{c}_{+}=(c_{1},\dots, c_{i-1})$ and write 
\[ \upalpha_{\vec{c}_{+}} = c_{1}\upalpha_{1} + \cdots + c_{i-1}\upalpha_{i-1} +\upalpha_{i}.\]
Note trivially that $|\upalpha_{\vec{c}_{+}}|=|\upalpha_{i}|$.  Then by part (1) of this Lemma,
\begin{align*}
 \left|  \Upomega_{i}^{\mathfrak{a}}(q^{n}-1)\right| = \left| \sum_{\vec{c}_{+}}\sum_{c} (c+\upalpha_{\vec{c}_{+}})^{1-q^{n}}\right|
 \leq \max \{ |\upalpha_{\vec{c}_{+}}|^{q^{n}(1-q)} \} =  |\upalpha_{i}|^{q^{n}(1-q)}.
\end{align*}
\end{proof}

In what follows, we write 
\[ \hat{\upzeta}^{\mathfrak{a}}(n(q-1)) = \upzeta^{\mathfrak{a}}(n(q-1))-1 .\]
By Lemma \ref{omegalemma} we have immediately

\begin{coro}\label{boundonzetaa}  Let $\mathfrak{a}$ be as above.  Then for all $n\in\N$, 
\[  \left|\hat{\upzeta}^{\mathfrak{a}}(q^{n}-1)\right|  =\left|  \Upomega_{1}^{\mathfrak{a}}(q^{n}-1)\right|
=|\upalpha_{1}|^{q^{n}(1-q)} <1.
\]
\end{coro}

\begin{theo}\label{jadifffrom1}  For all $\mathfrak{a}\subset A_{\infty_{1}}$ non principal, $j(\mathfrak{a})\not=j((1))$. 
\end{theo}

\begin{proof}  
It will be enough to prove
the Theorem with $\mathfrak{a}$ replaced by
the fractional ideal
$g^{-1}\mathfrak{a}$ studied above.
For any ideal $\mathfrak{b}$ we denote $\tilde{J}(\mathfrak{b})= \upzeta^{\mathfrak{b}}(q^{2}-1)/\upzeta^{\mathfrak{b}}(q-1)^{q+1}$.
It will suffice to show that $\tilde{J}(\mathfrak{a})\not= \tilde{J}((1))$ i.e. that the numerator of
\[  \tilde{J}(\mathfrak{a})- \tilde{J}((1)) =\frac{\upzeta^{\mathfrak{a}}(q^{2}-1)\cdot
\upzeta^{(1)}(q-1)^{q+1}-\upzeta^{\mathfrak{a}}(q-1)^{q+1}\cdot
\upzeta^{(1)}(q^{2}-1)}{ \left(\upzeta^{\mathfrak{a}}(q-1)\cdot \upzeta^{(1)}(q-1)\right)^{q+1} }\]
does not vanish.  This numerator can be written
\begin{align}\label{numerator} (\hat{\upzeta}^{\mathfrak{a}}(q^{2}-1)+1)(\hat{\upzeta}^{(1)}(q-1)^{q}+1)(\hat{\upzeta}^{(1)}(q-1)+1)- &  \\
 (\hat{\upzeta}^{(1)}(q^{2}-1)+1)(\hat{\upzeta}^{\mathfrak{a}}(q-1)^{q}+1)(\hat{\upzeta}^{\mathfrak{a}}(q-1)+1). \nonumber 
 \end{align}
 Developing the products, by Corollary \ref{boundonzetaa} we see that (\ref{numerator}) can be written
 \[ -\hat{\upzeta}^{\mathfrak{a}}(q-1)+\hat{\upzeta}^{(1)}(q-1)   + 
  \hat{\upzeta}^{\mathfrak{a}}(q^{2}-1)-\hat{\upzeta}^{(1)}(q^{2}-1)+\text{lower}. \]
 Therefore,  by Lemma \ref{omegalemma} and Corollary \ref{boundonzetaa}, the absolute value of (\ref{numerator}) is 
\begin{align*}  \left|\hat{\upzeta}^{\mathfrak{a}}(q-1) \right|=|\upalpha_{1}|^{q(1-q)} & >\max \left\{  |f|^{q(1-q)} ,|\upalpha_{1}|^{q^{2}(1-q)}  ,|f|^{q^{2}(1-q)}  \right\} \\
& = 
\max \left\{\left|\hat{\upzeta}^{(1)}(q-1) \right|, 
\left|\hat{\upzeta}^{\mathfrak{a}}(q^{2}-1)\right|,\left|\hat{\upzeta}^{(1)}(q^{2}-1) \right| 
\right\} \end{align*}
  and we are done.
 \end{proof}

 \begin{coro}\label{multicoro} $j^{\rm qt}$ is multivalued.
 \end{coro}
 
 \begin{proof} If $f$ satisfies $v_{\infty}(f)=-d<-1$, then $j((1)), j(\mathfrak{a}_{d-1})\in j(f)$ by Theorem \ref{jfidthem}, and are distinct by Theorem \ref{jadifffrom1}.
 \end{proof}
 
  For any fractional ideal $\mathfrak{c}$ with basis of the form specified by Lemma 2, we write
\[ j(\mathfrak{c}) =(T^{q}-T)^{q+2}\cdot \frac{\upzeta^{\mathfrak{c}}(q-1)^{q+1}}{\Updelta^{\mathfrak{c}}}\]
where
\[ \Updelta^{\mathfrak{c}} = U^{\mathfrak{c}}-V^{\mathfrak{c}}\]
with 
\begin{align*}
U^{\mathfrak{c}} & = \left( (T^{q}-T)(1+\Upomega_{1}^{\mathfrak{c}}(q-1)+\cdots ) \right)^{q+1} \\
& =T^{q(q+1)}-T^{q^{2}+1}-T^{2q}+T^{q+1}+T^{q(q+1)}\sum_{c} (c+\upalpha_{1})^{1-q} + \text{lower}
\end{align*}
(in the above, we use that $(T^{q}-T)^{q+1}=(T^{q^{2}}-T^{q})(T^{q}-T)$) and 
\begin{align*} V^{\mathfrak{c}} &= (T^{q}-T)(T^{q^{2}} -T)(1+\Upomega_{1}^{\mathfrak{c}}(q^{2}-1)+\cdots  ) \\
& = T^{q(q+1)}-T^{q^{2}+1}-T^{q+1}+T^{2}+T^{q(q+1)}\sum_{c} (c+\upalpha_{1})^{1-q^{2}} + \text{lower}.
\end{align*}
Thus \[ \Updelta^{\mathfrak{c}}= -T^{2q}+2T^{q+1}-T^{2} +T^{q(q+1)}\sum_{c} (c+\upalpha_{1})^{1-q} + \text{lower}.\]

\begin{lemm}\label{DeltaLemma}  Let $\mathfrak{c}$, $ \Updelta^{\mathfrak{c}}$ be as above.  Then
\[
|\Updelta^{\mathfrak{c}}| =\left\{ 
\begin{array}{ll}
q^{q+1} & \text{if  $|\upalpha_{1}|=q$} \\
q^{2q} & \text{otherwise.} \\
 \end{array}
\right. 
\]
\end{lemm}
 \begin{proof}  Suppose first that $|\upalpha_{1}|=q$.  Since $\upalpha_{1}$ is a quotient of {\it monic} polynomials in $f, fT,\dots$ and since $a$ (the linear coefficient in $f^{2}=af+b$) is monic in $T$, we may write $\upalpha_{1}=T+\updelta$
 where $|\updelta |<q$.  In particular, by Lemma  \ref{omegalemma}, item (1), we have
 \begin{align*} \left| \sum_{c} (c+\upalpha_{1})^{1-q} -\sum_{c} (c+T)^{1-q} \right| & = \left|\frac{(\upalpha_{1}^{q}-\upalpha_{1})\prod_{c}(c+T^{q})  
 -(T^{q}-T)\prod_{c}(c+\upalpha_{1}^{q})}{\prod_{c} \big( (c+\upalpha_{1}^{q})(c+T^{q})\big)}\right| < q^{q(1-q)}.   
 \end{align*}
 Therefore, we may write 
 \[ \Updelta^{\mathfrak{c}}= -T^{2q}+2T^{q+1}+T^{q(q+1)}\sum_{c} (c+T)^{1-q} + \text{lower}.\]
 By Lemma \ref{omegalemma}, item (1),
 \begin{align*}
  -T^{2q} +2T^{q+1}+ T^{q(q+1)}\sum_{c} (c+T)^{1-q}  & =  \\ 
  \frac{(-T^{2q}+2T^{q+1} )\cdot \prod_{c} (c+T^{q})+T^{q(q+1)}\cdot (T^{q}-T)}{\prod_{c} (c+T^{q})} & = \\
  \frac{T^{q(q+1)+1}+ \text{lower} }{\prod_{c} (c+T^{q})}.
  \end{align*}
It follows that,
 \[ \left| -T^{2q}+2T^{q+1}+T^{q(q+1)}\sum (c+T)^{1-q} \right|=q^{q+1} ,
 \]
 and we conclude that in this case,
 \[ | \Updelta^{\mathfrak{c}}|=q^{q+1}.\]  
  If $|\upalpha_{1}|>q$, by Lemma \ref{omegalemma}, item (1),
 \[  \left|T^{q(q+1)}\sum (c+\upalpha_{1})^{1-q}\right|=q^{q(q+1)}\cdot |\upalpha_{1}|^{q(1-q)}<q^{2q}=|T^{2q}| \]
 hence  \[ | \Updelta^{\mathfrak{c}}|=q^{2q}.\]
 \end{proof}

 Given $\mathfrak{b}$ denote by $\mathfrak{b}_{i}=\mathfrak{b}\mathfrak{a}_{i}$ and write
 \[ {\sf N}(j (\mathfrak{b})) : = \prod_{i=0}^{d-1} j (\mathfrak{b}_{i}) .\]
 
 \begin{theo}\label{normdiff} If $\mathfrak{b}\not\in [ \mathfrak{a}_{i}]$ for all $i$, then ${\sf N}(j (\mathfrak{b}))\not={\sf N}(j ((1)))$.
 \end{theo}
 
 \begin{proof} 
In the notation established above, we have for a constant $C\in A$,  \[  {\sf N}(j (\mathfrak{b})) =C^{d} \cdot \frac{\prod_{i=0}^{d-1}\upzeta^{\mathfrak{b}_{i}}(q-1)^{q+1}}{\prod_{i=0}^{d-1}\Updelta^{\mathfrak{b}_{i}} }.\]
Let us first analyze the absolute value of the quotient 
\begin{align}\label{normquotient} \left| \frac{ {\sf N}(j ((1)))}{ {\sf N}(j (\mathfrak{b}))}\right| =\left|\frac{\prod_{i=0}^{d-1}\Updelta^{\mathfrak{b}_{i}} }{\prod_{i=0}^{d-1}\Updelta^{\mathfrak{a}_{i}} }\right|,  \end{align}
where in (\ref{normquotient}) we use the fact that for normalized $\mathfrak{c}$,   $|\upzeta^{\mathfrak{c}}(q-1)|=1$. 
For $i\not=0$,
\[ U^{\mathfrak{a}_{i}}= T^{q(q+1)}-T^{q^{2}+1}-T^{2q}+T^{q+1}+T^{q(q+1)}\sum (c+T)^{1-q} + \text{lower}\]
and
\[ V^{\mathfrak{a}_{i}}= T^{q(q+1)}-T^{q^{2}+1}-T^{q+1}+T^{2}+T^{q(q+1)}\sum (c+T)^{1-q^{2}} + \text{lower}\]
so 
\[ \Updelta^{\mathfrak{a}_{i}}= -T^{2q}+2T^{q+1} +T^{q(q+1)}\sum (c+T)^{1-q} + \text{lower}.\]
 By Lemma \ref{DeltaLemma}, for $i\not=0$,
 \[ | \Updelta^{\mathfrak{a}_{i}}|=q^{q+1}.\]
 On the other hand, 
\[ \Updelta^{(1)}= -T^{2q}+2T^{q+1}+T^{q(q+1)}\sum (c+f)^{1-q} + \text{lower}\] and Lemma \ref{DeltaLemma} gives this time
\[  | \Updelta^{(1)}|=q^{2q}.  \]
Thus the absolute value of the denominator in (\ref{normquotient}) is $q^{(d-1)(q+1)+2q}$.  If the absolute value of  (\ref{normquotient}) is not 1, we are done.  So suppose 
otherwise i.e.\ that the numerator also has absolute value $q^{(d-1)(q+1)+2q}$. 
By Lemma \ref{DeltaLemma}, in order for the numerator to have absolute value $q^{(d-1)(q+1)+2q}$, it must be the case that: \begin{itemize}
\item[i.] For some $i_{0}$, $|\Updelta^{\mathfrak{b}_{i_{0}}}|=q^{2q}$.
\item[ii.] For all $i\not=i_{0}$, $|\Updelta^{\mathfrak{b}_{i}}|=q^{q+1}$.
\end{itemize}
Without loss of generality, we may assume that $i_{0}=0$ and denote by $\upalpha=\upalpha_{0,1},\upalpha_{(i)}=\upalpha_{i,1}$ the smallest degree basis element 
of $\mathfrak{b}$,$\mathfrak{b}_{i}$  not equal to $1$, $i\geq 1$.
For i.\ to occur, we must have that $|\upalpha|>q$.
For ii.\ to occur, we must have that $|\upalpha_{(i)}|=q$ for $i\geq 1$.
Hence, writing $\upalpha_{(i)}=T+\updelta_{(i)}$ as in the proof of Lemma \ref{DeltaLemma}, for $i\not=0$, we have
\[ \Updelta^{\mathfrak{b}_{i}} =-T^{2q}+2T^{q+1}+T^{q(q+1)}\sum (c+T)^{1-q} + \text{lower}=: \Updelta + \text{lower},\quad i\not=0  \]
as well as
\[ \Updelta^{\mathfrak{a}_{i}} =-T^{2q}+2T^{q+1} +T^{q(q+1)}\sum (c+T)^{1-q} + \text{lower}=: \Updelta + \text{lower}, \quad i\not=0 .\]
We also write
\[\Updelta^{\mathfrak{b}} = -T^{2q}+2T^{q+1}+T^{q(q+1)}\sum (c+\upalpha)^{1-q} + \text{lower}:= \Updelta_{\upalpha} + \text{lower}\]
and
\[\Updelta^{(1)} = -T^{2q}+2T^{q+1} +T^{q(q+1)}\sum (c+f)^{1-q} + \text{lower}:= \Updelta_{f} + \text{lower}.\]
Now the difference ${\sf N}(j ((1)))-{\sf N}(j (\mathfrak{b}))$ may be written as a fraction whose numerator is, up to a multiplicative constant, given by
\begin{align*}
\prod_{i=0}^{d-1} \Updelta^{\mathfrak{b}_{i}}\prod_{i=0}^{d-1} \upzeta^{\mathfrak{a}_{i}}(q-1)^{q+1}-
\prod_{i=0}^{d-1} \Updelta^{\mathfrak{a}_{i}}\prod_{i=0}^{d-1} \upzeta^{\mathfrak{b}_{i}}(q-1)^{q+1} & = \\
(\Updelta_{\upalpha} + \text{lower})\prod_{i\not=0}(\Updelta + \text{lower})(1+\text{lower}) -(\Updelta_{f} + \text{lower})\prod_{i\not=0}(\Updelta + \text{lower})(1+\text{lower}) & = \\
\Updelta^{d-1}\cdot \left\{ \Updelta_{\upalpha}-\Updelta_{f}\right\}  + \text{lower}& =\\
\Updelta^{d-1}T^{q(q+1)}\cdot \left\{ \sum (c+\upalpha)^{1-q}- \sum (c+f)^{1-q}\right\} + \text{lower}  &.
\end{align*}
Since $|\upalpha|<|f|$, Lemma \ref{omegalemma} gives
\[  \left|   \sum (c+\upalpha)^{1-q}\right| = |\upalpha|^{q(1-q)}>|f|^{q(1-q)}=   \left|   \sum (c+f)^{1-q}\right|. \]
It follows that the numerator of ${\sf N}(j (\mathfrak{b}))-{\sf N}(j ((1)))$ has absolute value
\[ |C||\Updelta|^{d-1}q^{q(q+1)}   |\upalpha|^{q(1-q)}\not=0,\] where $C\not=0$ is a constant, and we are done.
 \end{proof}

\section{Generation of the Hilbert Class Field}\label{HCFGenerationSection}

Let $H_{A_{\infty_{1}}}$ be the Hilbert class field associated to the ring $A_{\infty_{1}}$: the maximal abelian unramified extension of $K$ which splits
completely at $\infty_{1}$, see \cite{Ros}.  Since $K$ is totally real, the constant field of $K$ = $\F_{\infty}=\F_{q}$.  This means that $H_{A_{\infty_{1}}}=
H_{A_{\infty_{1}}}^{+}$ = {\it narrow} Hilbert class field associated to $A_{\infty_{1}}$, see Definition 7.4.1 and Proposition 7.4.10 of \cite{Goss}.

\begin{theo}\label{HAinft1Thm} Let $\mathfrak{a}\subset K$ be an $A_{\infty_{1}}$ fractional ideal.  Then $j(\mathfrak{a})\in \bar{k}$ and \[ H_{A_{\infty_{1}}}=K(j(\mathfrak{a})).\]
\end{theo}

\begin{proof} 
By Goss' Lemma (Lemma 8.18.1 of \cite{Goss} or Theorem 5.2.5 of \cite{Thak}), there exists a constant $\upxi\in {\bf C}_{\infty}$ such that
\[ \upzeta^{\mathfrak{a}}(n(q-1))/\upxi^{n(q-1)}\in H_{A_{\infty_{1}}}.\] Since 
\[ 
J(\mathfrak{a}) =  \frac{T^{q^{2}}-T}{(T^{q}-T)^{q+1}} \cdot
\frac{\upzeta^{\mathfrak{a}}(q^{2}-1)}{\upzeta^{\mathfrak{a}}(q-1)^{q+1}}=\frac{T^{q^{2}}-T}{(T^{q}-T)^{q+1}} \cdot
\frac{\upzeta^{\mathfrak{a}}(q^{2}-1)/\upxi^{q^{2}-1}}{\left(\upzeta^{\mathfrak{a}}(q-1)/\upxi^{q-1}\right)^{q+1}},
\]
 it follows that $j(\mathfrak{a})\in H_{A_{\infty_{1}}}$ as well.
The elements of the set 
\[  \left\{  \upzeta^{\mathfrak{a}}(n(q-1))/\upxi^{n(q-1)}:\; [\mathfrak{a}]\in{\sf Cl}_{A_{\infty_{1}}} \right\}  \]
are conjugate by the action ${\rm Gal}(H_{A_{\infty_{1}}}/K)$.  See the proof of Theorem 8.19.4 of \cite{Goss}.
Moreover, if $\upsigma_{\mathfrak{b}}\in {\rm Gal}(H_{A_{\infty_{1}}}/K)$ corresponds to $[\mathfrak{b}]\in {\sf Cl}_{A_{\infty_{1}}}$
by the Artin reciprocity map, then the fundamental theorem of Hayes theory (Theorem 7.4.8 of \cite{Goss} or Theorem 14.7 of \cite{Hayes}) implies
that 
\[\left(\upzeta^{\mathfrak{a}}(n(q-1))/\upxi^{n(q-1)}\right)^{\upsigma_{\mathfrak{b}}}=\upzeta^{\mathfrak{b}^{-1}\mathfrak{a}}(n(q-1))/\upxi^{n(q-1)}, \]
see \cite{Shu}, Theorem A.6 and its proof. Therefore we conclude that
\[ j(\mathfrak{a})^{\upsigma_{\mathfrak{b}}}=j(\mathfrak{b}^{-1}\mathfrak{a}). \]
Since multiplication by $\mathfrak{b}^{-1}$ acts transitively on ${\sf Cl}_{A_{\infty_{1}}}$, by Theorem \ref{jadifffrom1} we see that for all $[\mathfrak{a}]\not=[\mathfrak{a}']$, 
$j(\mathfrak{a})\not=j(\mathfrak{a}')$.  Thus the size of the Galois orbit of $j(\mathfrak{a})$ is $h_{A_{\infty_{1}}}$ = the class number of $A_{\infty_{1}}$ 
= $\# {\rm Gal}(H_{A_{\infty_{1}}}/K)$.  This proves that $j(\mathfrak{a})$ Galois generates $H_{A_{\infty_{1}}}$.  However, since $H_{A_{\infty_{1}}}/K$
is abelian, it follows that $K(j(\mathfrak{a}))$ is Galois, hence $K(j(\mathfrak{a}))$ contains all of the conjugates of $j(\mathfrak{a})$ and so $K(j(\mathfrak{a}))=H_{A_{\infty_{1}}}$.
\end{proof}

From the proof of Theorem \ref{HAinft1Thm} we have the injectivity of $j$ on ideal classes:

\begin{coro}\label{injectivitycoro} $j:{\sf Cl}_{A_{\infty_{1}}}\rightarrow \bar{k}$ is injective.  In particular, $\# j^{\rm qt}(f)=d$.
\end{coro}

By Theorem \ref{jfidthem} we have immediately:

\begin{coro}  $j^{\rm qt}(f)\subset \bar{k}$ and $H_{A_{\infty_{1}}}=K(j^{\rm qt}(f))$.
\end{coro}

We now turn to consideration of the {\it relative} Dedekind domain
\[ \mathcal{O}_{K} =\text{ integral closure of $A$ in $K$}.\]  Then $\mathcal{O}_{K}$ is the ring of functions regular outside of $\uppi^{-1}(\infty)=\{ \infty_{1},\infty_{2}\}\subset\Upsigma$.
 Writing $\langle S\rangle_{A}$ for the $A$-module generated by the set $S\subset K$,
we may identify 
\[ \mathcal{O}_{K} =\langle 1,f\rangle_{A}= \F_{q}[T,f]\supset A_{\infty_{1}}. \]
Note that $\mathcal{O}_{K}=f^{-1}A_{\infty_{1}}$.
By Dirichlet's Unit Theorem \cite{cohn}, the unit group $\mathcal{O}_{K}^{\ast}$ is generated by the $\F_{q}^{\ast}$ multiples of $f$.  The inflation map $\mathfrak{a}\mapsto \mathfrak{a}\mathcal{O}_{K}$, $\mathfrak{a}\subset A_{\infty_{1}}$ an ideal, induces a surjective
homomorphism of class groups
\[ \Upphi:{\sf Cl}_{A_{\infty_{1}}}\longrightarrow {\sf Cl}_{\mathcal{O}_{K}} . \]

\begin{prop}\label{KerProp}  ${\rm Ker}(\Upphi )= \langle [\mathfrak{a}_{d-1}]\rangle$, a cyclic subgroup of ${\sf Cl}_{A_{\infty_{1}}}$ of order $d$.  \end{prop}

\begin{proof}  Let $\mathcal{D}$ be the divisor group of $K$, $\mathcal{P}$ the subgroup of principal divisors and ${\sf Cl}_{K}=\mathcal{D}/\mathcal{P}$ the divisor class group.  
Then by Lemma 1.1 of \cite{Ros}, we have 
\[  {\sf Cl}_{A_{\infty_{1}}} \cong {\sf Cl}_{K}/\langle \infty_{1}\rangle,\quad {\sf Cl}_{\mathcal{O}_{K}} \cong {\sf Cl}_{K}/\langle \infty_{1},\infty_{2}\rangle\]
and $\Upphi$ can be identified with the canonical projection 
\[  {\sf Cl}_{K}/\langle \infty_{1}\rangle\longrightarrow   {\sf Cl}_{K}/\langle \infty_{1},\infty_{2}\rangle.\]
Thus ${\rm Ker}(\Upphi )$ is the cyclic group $\langle \infty_{2}\rangle$ generated by the class of $\infty_{2}$ in $ {\sf Cl}_{K}/\langle \infty_{1}\rangle$.
Now the isomorphism $t:{\sf Cl}_{K}/\langle \infty_{1}\rangle \cong {\sf Cl}_{A_{\infty_{1}}}$ is induced by the association 
\[  \sum_{P\not=\infty_{1}} n_{P}(P)\longrightarrow \prod\mathfrak{m}_{P}^{n_{P}} \]
where $\mathfrak{m}_{P}\subset A_{\infty_{1}}$ is the maximal ideal of functions having a zero at $P\not=\infty_{1}$.  It follows that ${\rm Ker}(\Upphi )$ is generated
by the ideal class of $\mathfrak{m}_{\infty_{2}}$.  However
\[  \mathfrak{a}_{d-1}  = (f,fT,\dots ,fT^{d-1}) \subset \mathfrak{m}_{\infty_{2}},\]
since $f$ has a zero of order $d$ at $\infty_{2}$ and $T$ has a pole of order $1$ at $\infty_{2}$,
so $fT^{i}$, $i=0,\dots ,d-1$, all vanish at $\infty_{2}$.  But $\mathfrak{a}_{d-1}$ is maximal since $A_{\infty_{1}}/\mathfrak{a}_{d-1}\cong\F_{q}$, so $ \mathfrak{a}_{d-1}=\mathfrak{m}_{\infty_{2}}$.
\end{proof}

\begin{coro}\label{classnumberform} Let $h_{K}$, $h_{A_{\infty_{1}}}$ and
$h_{\mathcal{O}_{K}}$ be the class numbers of $K$, $A_{\infty_{1}}$ and $\mathcal{O}_{K}$, respectively.  Then 
\[ h_{K}=h_{A_{\infty_{1}}} = h_{\mathcal{O}_{K}}\cdot d.\]
\end{coro}
\begin{proof} Since $K\subset k_{\infty}$, its field of constants is $\F_{q}$.  Hence the degree $d_{\infty_{1}}$ of $\infty_{1}$ is $1$, giving the first equality.  The second
equality follows from Proposition \ref{KerProp}.
\end{proof}

Let $H_{\mathcal{O}_{K}}$ be the Hilbert class field associated to the ring $\mathcal{O}_{K}$: the maximal abelian unramified extension of $K$ which splits
completely at $\infty_{1}$ and $\infty_{2}$.  See again \cite{Ros}. Clearly $H_{A_{\infty_{1}}}\supset H_{\mathcal{O}_{K}}$ is a Galois extension having Galois group 
$\cong {\rm Ker}(\Upphi )=\langle[\mathfrak{a}_{d-1}]\rangle$.  
Denote 
\[   {\sf N} (j^{\rm qt}(f)) = \text{product of elements in $j^{\rm qt}(f)$} = \prod_{i=0}^{d-1}j(\mathfrak{a}_{i})={\sf N}((1)) .
\]

\begin{theo}\label{maintheo}  $H_{\mathcal{O}_{K}}=K( {\sf N}(j^{\rm qt}(f)))$.
\end{theo}

\begin{proof} Let ${\rm Norm}:H_{A_{\infty_{1}}}\longrightarrow H_{\mathcal{O}_{K}}$ be the norm map.  Then ${\rm Norm}(j((1))) =  {\sf N} (j^{\rm qt}(f))$, since each 
$\upsigma\in {\rm Gal}(H_{A_{\infty_{1}}}/H_{\mathcal{O}_{K}} )$ corresponds via reciprocity to some $\mathfrak{a}_{i}$ with
\[   j((1))^{\upsigma}=j(\mathfrak{a}_{i}^{-1}). \]  
By Theorem \ref{normdiff} and Corollary \ref{classnumberform}, the set 
\[  \{ {\rm Norm}(j(\mathfrak{b})):\; \mathfrak{b}\in{\sf Cl}_{A_{\infty_{1}}} \} \ni {\sf N}(j^{\rm qt}(f))  \] consists of $[H_{\mathcal{O}_{K}}:K]$ elements which form an orbit
with respect to the action of ${\rm Gal}(H_{\mathcal{O}_{K}}/K)$.   Since $H_{\mathcal{O}_{K}}/K$ is abelian, all of its subextensions are Galois, in particular, $K( {\sf N}(j^{\rm qt}(f)))/K$
is a Galois subextension, of degree $[H_{\mathcal{O}_{K}}:K]$.  
The Theorem now follows.
\end{proof}

 \end{document}